\documentclass[letterpaper]{article} %

\usepackage[utf8]{inputenc} 
\usepackage[T1]{fontenc}    
\usepackage{hyperref}       
\usepackage{url}            
\usepackage{booktabs}       
\usepackage{amsfonts}       
\usepackage{microtype}      
\usepackage{color,xcolor}         

\usepackage{natbib} 
\usepackage{mathtools}
\newcommand{\eqdef}{\coloneqq}

\usepackage{fullpage}

\usepackage{mathtools}   
\usepackage{algorithm}
\usepackage{algorithmic}

\usepackage{xspace}
\usepackage{scalefnt}

\usepackage{amsmath, amsfonts,amssymb,amsthm}
\usepackage{cite}
\DeclareMathOperator*{\argmin}{arg\,min}

\newcommand{\sqn}[1]{\left\| #1 \right\|^2}
\newcommand{\Exp}[1]{\mathbb{E}\!\left[ #1 \right]} 
\newcommand{\Expb}[2]{\mathbb{E}\!\left[#1\ \vert\ #2\right]}
\newcommand{\Expc}[1]{\Expb{#1}{\mathcal{F}^t}}

\newtheorem{theorem}{Theorem}
\newtheorem{proposition}{Proposition}
\newtheorem{corollary}{Corollary}

\usepackage{xspace}
\usepackage{scalefnt}
\definecolor{mydarkred}{rgb}{0.8,0.0,0.0}
\newcommand{\algn}[1]{{\sf\color{mydarkred}\scalefont{0.955}{#1}}\xspace}
\newcommand{\algno}{\algn{Point-SAGA}}

\title{\textbf{A Simple Linear Convergence Analysis \\of the Point-SAGA Algorithm}}

\author{%
 Laurent Condat \&  Peter Richt\'{a}rik\\
Computer Science Program, CEMSE Division,\\ King Abdullah University of Science and Technology (KAUST)\\ Thuwal, 23955-6900, Kingdom of Saudi Arabia\\
 \& SDAIA-KAUST Center of Excellence in Data Science and \\Artificial Intelligence 
(SDAIA-KAUST AI)\\
Contact: see \texttt{https://lcondat.github.io/} 
}

\date{May 30, 2024}

\begin{document}

\maketitle

\begin{abstract}
Point-SAGA is a randomized algorithm for minimizing a sum of convex functions using their proximity operators (proxs), proposed by \citet{def16}. At every iteration, the prox of only one randomly chosen function is called. We generalize the algorithm to any number of prox calls per iteration, not only one, and propose a simple proof of linear convergence when the functions are smooth and strongly convex. 
\end{abstract}

\section{Introduction}

Optimization problems arise in machine learning, signal and image processing, control, statistics, and many other fields \citep{pal09, sra11, bac12, cev14, pol15, bub15, glo16, cha16, sta16}. 
Stochastic Gradient Descent algorithms are well suited to minimize a sum of many smooth functions. They proceed by calling at every iteration the gradients of some randomly chosen functions, instead of all of them
\citep{Bottou2012,gow20a,gor202}. 
On the other hand, proximal algorithms make calls to the proximity operators (proxs) of the functions, instead of their gradients \citep{com10,par14,bec17,con19,com21,con22}. We recall that for any function $\phi$, the prox of $ \phi$ is 
 $
 \mathrm{prox}_{\phi}: x\in\mathcal{X} \mapsto \argmin_{x'\in\mathcal{X}}\big(\phi(x') + \frac{1}{2} \|x'-x\|^2\big)
  $ \citep{bau17}. 
  This operator has a closed form for many functions of practical interest \citep{par14,pus17,ghe18}, see also the website \url{http://proximity-operator.net}.

The optimization problem we consider is 
 the minimization of a sum of $n\geq 1$ convex functions $f_i$, $i\in[n]\eqdef\{1,\ldots,n\}$, over a finite-dimensional  real Hilbert space 
$\mathcal{X}$:
\begin{equation}
\mathrm{Find} \ x^\star \in \argmin_{x\in\mathcal{X}}  \sum_{i=1}^n f_i(x).\label{eqpb0}
 \end{equation}

We assume that for every $i\in [n]$, $f_i$ is $\mu$-strongly convex, for some  $\mu> 0$; that is $f_i-\frac{\mu}{2}\|\cdot\|^2$ is convex. Hence, the solution $x^\star$ of \eqref{eqpb0} exists and is unique. 

We also assume that for every $i\in [n]$, $f_i$ is $L$-smooth, for some $L\geq \mu$; that is, $f_i:\mathcal{X}\rightarrow \mathbb{R}$ is differentiable and its gradient is $L$-Lipschitz continuous: 
 for every $(x,y)\in\mathcal{X}^2$, 
\begin{equation*}
\|\nabla f_i(x)-\nabla f_i(y)\|\leq L \|x-y\|.\label{eqdef1}
\end{equation*}
  
\algno is a randomized proximal algorithm to solve \eqref{eqpb0}, which calls at every iteration the prox of only one randomly chosen  function $f_i$ \citep{def16}. In this paper, we generalize \algno to a minibatch version with any number $s\in[n]$ of activated functions per iteration, not only $s=1$. 
We prove its linear convergence in the considered setting where the functions are smooth and strongly convex. Our proof is simple and short, and our rate when $s=1$ is better than the one in \citet{def16}.

We leave it for future work to extend the setting to functions with different smoothness and strong convexity constants $L_i$ and $\mu_i$, along with random sampling strategies more general than uniform sampling. It would also be interesting to study convergence  in the general convex case where the functions $f_i$ are neither smooth nor strongly convex. A randomized primal--dual algorithm for a more general class of problems has been proposed in \citet{con22rp}, with a convergence analysis under different assumptions. \algno has been studied under a similarity assumption different from smoothness in \citet{sad24}.

 \begin{figure*}[t]	
\begin{algorithm}[H]
		\caption{\algno}\label{alg1}
		\begin{algorithmic}[1]
			\STATE  \textbf{input:} initial estimate $x^0\in\mathcal{X}$, initial gradient estimates $(g_i^0)_{i=1}^n\in\mathcal{X}^n$;
			\STATE stepsize $\gamma>0$; minibatch size $s\in [n]$
			\STATE $g^0\coloneqq \frac{1}{n}\sum_{i=1}^n g_i^0$
			\FOR{$t=0, 1, \ldots$}
			\STATE pick a subset $\Omega^t\subset [n]$ of size $s$ uniformly at random
			\FOR{$i\in \Omega^t$}
			\STATE $z_i^{t}\coloneqq x^t + \gamma g_i^t-\gamma g^t $
			\STATE $x_i^{t+1}\coloneqq \mathrm{prox}_{\gamma f_i} (z_i^t)$
			\STATE $g_i^{t+1}\coloneqq \frac{1}{\gamma} (z_i^t - x_i^{t+1})\quad//=\nabla f_i(x_i^{t+1})$
			\ENDFOR
			\FOR{$i\in[n]\backslash\Omega^t$} 
			\STATE $g_i^{t+1}\coloneqq g_i^t$
			\ENDFOR
			\STATE $x^{t+1}\eqdef \frac{1}{s}\sum_{i\in \Omega^t} x_i^{t+1}$
			\STATE $g^{t+1}\coloneqq \frac{n-s}{n}g^t + \frac{s}{n\gamma}(x^t-x^{t+1})\quad//=\frac{1}{n} \sum_{i=1}^n g_i^{t+1}$
			\ENDFOR
		\end{algorithmic}
		\end{algorithm}\end{figure*}

\section{The \algno algorithm}

Our proposed generalized version of \algno is shown as Algorithm~\ref{alg1}. It proceeds as follows. At every iteration $t\geq 0$, a subset $\Omega^t\subset [n]$ of size $s\in[n]$ is chosen uniformly at random (\citet{def16} only considered the case $s=1$). Then the proxs of the functions $f_i$ are called, for every $i\in\Omega^t$. Their outputs $x_i^{t+1}$ are averaged to form the next estimate $x^{t+1}$. 
A table $(g_i)_{i=1}^n$ of estimates of the gradients $\nabla f_i(x^\star)$ is maintained. $g_i^{t+1}$ is updated to  $\nabla f_i(x_i^{t+1})$ for every $i\in\Omega^t$ and the other gradient estimates remain unchanged. There is actually no need to evaluate any gradient: an important property of the prox is that $x_i^{t+1}=\mathrm{prox}_{\gamma f_{i}}(z_i^t)$ satisfies \citep{bau17}
\begin{equation}
x_i^{t+1}+\gamma \nabla f_{i}(x_i^{t+1})=z_i^t,\label{eqp1}
\end{equation}
so that $\nabla f_{i}(x_i^{t+1})$ is equal to $(z_i^t - x_i^{t+1})/\gamma$. The average $g^t = \frac{1}{n} \sum_{i=1}^n g_i^{t}$ is maintained. Updating $g^{t+1}$ by explicitly computing $\sum_{i\in \Omega^t} g_i^{t+1}$ could be costly if $s$ is large. Instead, there is a simple way to compute $g^{t+1}$:
\begin{proposition}\label{lem1}
In \algno, for every $t\geq 0$, $g^{t+1}\coloneqq \frac{n-s}{n}g^t + \frac{s}{n\gamma}(x^t-x^{t+1})$ is equal to $\frac{1}{n} \sum_{i=1}^n g_i^{t+1}$. 
\end{proposition}
\begin{proof}
We prove the property by recurrence: let $t\geq 0$ and let us assume that $g^t=\frac{1}{n} \sum_{i=1}^n g_i^{t}$. Then
\begin{align*}
\frac{1}{n} \sum_{i=1}^n g_i^{t+1}&=g^t + \frac{1}{n} \sum_{i\in\Omega^t} g_i^{t+1}-g_i^t=g^t + \frac{1}{n} \sum_{i\in\Omega^t} \frac{1}{\gamma} (z_i^t - x_i^{t+1})-g_i^t\\
&=g^t + \frac{1}{n} \sum_{i\in\Omega^t} \frac{1}{\gamma} (x^t -\gamma g^t + \gamma g_i^t - x_i^{t+1})-g_i^t=g^t + \frac{s}{n\gamma } x^t - \frac{s}{n} g^t- \frac{1}{\gamma n} \sum_{i\in\Omega^t} x_i^{t+1}\\
&= \frac{n-s}{n}g^t + \frac{s}{n\gamma } x^t - \frac{s}{n\gamma } x^{t+1}=g^{t+1}.
\end{align*}
Since $g^0$ is set as $\frac{1}{n}\sum_{i=1}^n g_i^0$, the property is proved.
\end{proof}

\section{Linear Convergence Result}

\begin{theorem}
    \label{th1}
  In \algno with any stepsize $\gamma>0$, minibatch size $s\in[n]$, and initial estimates $x^0$, $g^0_1$,\ldots,$g^0_n \in \mathcal{X}$, we have, for every $t\geq 0$, \begin{equation}
\Exp{\Psi^t}\leq \max\left\{
1-\frac{1}{1 + \frac{L + \mu}{2\gamma\mu L}}
,1-\frac{1}{1+\gamma \frac{L+\mu}{2}}\frac{s}{n}
\right\}^k\Psi^0,\label{eqr1}
\end{equation}
where
\begin{align}
\Psi^t&\eqdef \left(1+\frac{2\gamma\mu L}{L+\mu}\right)s\Expc{\sqn{x^{t+1}-x^\star}}+\left(1+\frac{2}{\gamma(L+\mu)}\right)\gamma^2
\Expc{ \sum_{i=1}^n \sqn{g_i^{t+1}-\nabla f_i(x^\star)}}.
\end{align}
In addition, $(x^t)_{t\in\mathbb{N}}$ converges  to $x^\star$ and, for every $i\in[n]$, $(g_i^t)_{t\in\mathbb{N}}$ converges  to $\nabla f_{i}(x^\star)$,
almost surely.
\end{theorem}

\begin{corollary}
The iteration complexity of \algno, i.e., the  number of iterations to achieve $\Exp{\Psi^t}\leq \epsilon$ for any $\epsilon>0$, is 
\begin{equation}
\mathcal{O}\left(\left(\frac{1}{\gamma\mu }+\frac{(\gamma L +1)n}{s}\right)\log\left(\frac{\Psi^0}{\epsilon}\right)\right).
\end{equation}
Therefore, by choosing 
\begin{equation}
\gamma = \sqrt{\frac{s}{L\mu n}},
\end{equation}
the complexity of \algno is 
\begin{equation}
\mathcal{O}\left(\left(\sqrt{\frac{L n}{\mu s}}+\frac{n}{s}\right)\log\left(\frac{\Psi^0}{\epsilon}\right)\right).
\end{equation}
\end{corollary}
Thus,  \algno is an accelerated algorithm, as its iteration complexity with the appropriate stepsize depends on $\sqrt{L/\mu}$. 
If \algno  is run on a computing architecture with $m\geq 1$ machines computing the proxs in parallel, its computational complexity   is $\max(s/m,1)$ times the iteration complexity. Thus, $s=m$ is the best choice. 

We note that in the case $s=1$, our rate in \eqref{eqr1}  is smaller than the rate $\max\left\{\frac{1}{1+\gamma\mu},1-\frac{1}{(\gamma L+1)n}\right\}$ that appears implicitly in the proof of \citet[Theorem 5]{def16}, which only states linear convergence for a particular and complicated value of $\gamma$.  
We also note that when $s=n$, \algno is deterministic and reverts to the Douglas--Rachford algorithm applied to the problem $\min_{(x_1,\ldots,x_n)}\sum_{i=1}^n f_i(x_i)$ s.t.\ $x_1=\cdots=x_n$ \citep[Section 9]{con19}. When minimizing the sum of two convex functions, one of which is $L$-smooth and $\mu$-strongly convex, this algorithm has a rate $\max\left\{\frac{1}{1+\gamma\mu},1-\frac{1}{\gamma L+1}\right\}$, which is tight in general \citep[Theorems 2 and 3]{gis17}. It is interesting that for the particular problem we consider,  a better rate can be obtained.

\section{Proof of Theorem~\ref{th1}}

For every $i\in[n]$, we define $z_i^\star\eqdef x^\star+\gamma \nabla f_{i}(x^\star)$. 
Let $t\geq 0$ and $i\in\Omega^t$. According to \eqref{eqp1}, we have
\begin{equation*}
x_i^{t+1}-x^\star+\gamma \nabla f_{i}(x_i^{t+1})-\gamma \nabla f_{i}(x^\star)=z_i^t-z_{i}^\star.
\end{equation*}
Therefore,
\begin{align}
\sqn{x_i^{t+1}-x^\star+\gamma \nabla f_{i}(x_i^{t+1})-\gamma \nabla f_{i}(x^\star)}
&=\sqn{ x_i^{t+1}-x^\star}+\gamma^2\sqn{\nabla f_{i}(x_i^{t+1})-\nabla f_{i}(x^\star)}\notag\\
&\quad+2\gamma \left\langle \nabla f_{i}(x_i^{t+1})- \nabla f_{i}(x^\star),x_i^{t+1}-x^\star\right\rangle\label{eq27}\\
&=\sqn{ z_i^t- z_i^\star}.\notag
\end{align}
According to \citet[Lemma 3.11]{bub15}, 
since $f_i$ is $\mu$-strongly convex and $L$-smooth, it satisfies
\begin{align*}
\left\langle \nabla f_i(x_i^{t+1})- \nabla f_i(x^\star),x_i^{t+1}-x^\star\right\rangle &\geq \frac{\mu L}{L+\mu}\sqn{x_i^{t+1}-x^\star}+\frac{1}{L+\mu}\sqn{ \nabla f_i(x_i^{t+1})- \nabla f_i(x^\star)}.
\end{align*}
Using this inequality in \eqref{eq27}, we obtain
\begin{align}
\left(1+\frac{2\gamma\mu L}{L+\mu}\right)\sqn{ x_i^{t+1}-x^\star}+\left(\gamma^2+\frac{2\gamma}{L+\mu}\right)
\sqn{ \nabla f_i(x_i^{t+1})- \nabla f_i(x^\star)}
&\leq \sqn{ z_i^t- z_i^\star}.\label{eq40}
\end{align}
Replacing $ \nabla f_i(x_i^{t+1})$ by $g_i^{t+1}$ in \eqref{eq40} and summing over $i\in\Omega^t$ yields
\begin{align}
\left(1+\frac{2\gamma\mu L}{L+\mu}\right)\sum_{i\in\Omega^t}\sqn{ x_i^{t+1}-x^\star}+\left(1+\frac{2}{\gamma(L+\mu)}\right)\gamma^2\sum_{i\in\Omega^t}
\sqn{g_i^{t+1}- \nabla f_i(x^\star)}
&\leq\sum_{i\in\Omega^t} \sqn{ z_i^t- z_i^\star}.\label{eq41}
\end{align}
We now take expectations to remove the dependence on the random subset $\Omega^t$. We denote by $\mathcal{F}^t$ the $\sigma$-algebra generated by the collection of random variables $(x^0,g^0_1,\ldots,g^0_n,\ldots, x^t,g^t_1,\ldots,g^t_n)$.

We first evaluate the conditional expectation of the right-hand side of \eqref{eq41}:
\begin{align}
\Expc{ \sum_{i\in\Omega^t}\sqn{ z_i^t- z_{i}^\star}}&=\frac{s}{n} \sum_{i=1}^n \sqn{ x^t -x^\star - \gamma g^t + \gamma g_i^t - \gamma \nabla f_i(x^\star)}\notag\\
&=s\sqn{ x^t -x^\star}+\frac{s}{n}\gamma^2 \sum_{i=1}^n \sqn{   g_i^t -\nabla f_i(x^\star)- g^t }\notag\\
&\quad -\frac{s\gamma}{n} \left\langle x^t -x^\star,  \underbrace{\sum_{i=1}^n (g_i^t -\nabla f_i(x^\star)- g^t)}_0\right\rangle\notag\\
&=s\sqn{ x^t -x^\star}+\frac{s}{n}\gamma^2 \sum_{i=1}^n \sqn{   g_i^t -\nabla f_i(x^\star) }-s\gamma^2\sqn{g^t}\notag\\
&\leq s\sqn{ x^t -x^\star}+\frac{s}{n}\gamma^2 \sum_{i=1}^n \sqn{   g_i^t -\nabla f_i(x^\star) }.\label{eql1}
\end{align}
We now look at the conditional expectation of the left-hand side of \eqref{eq41}. First, we note that
 \begin{equation*}
\sqn{x^{t+1}-x^\star}=\sqn{\frac{1}{s}\sum_{i\in \Omega^t} (x_i^{t+1}-x^\star)}\leq \frac{1}{s}\sum_{i\in \Omega^t} \sqn{ x_i^{t+1}-x^\star},
\end{equation*}
so that
 \begin{equation}
s\Expc{\sqn{x^{t+1}-x^\star}}\leq \Expc{\sum_{i\in \Omega^t} \sqn{ x_i^{t+1}-x^\star}}.\label{eqa21}
\end{equation}
Second,
 \begin{align*}
\Expc{ \sum_{i=1}^n \sqn{g_i^{t+1}-\nabla f_i(x^\star)}}&= 
\Expc{ \sum_{i\in\Omega^t} \sqn{g_i^{t+1}-\nabla f_i(x^\star)}}+ \Expc{ \sum_{i\in[n]\backslash\Omega^t} \sqn{g_i^{t}-\nabla f_i(x^\star)}}\\
&=\Expc{ \sum_{i\in\Omega^t} \sqn{g_i^{t+1}-\nabla f_i(x^\star)}}+\frac{n-s}{n}\sum_{i=1}^n \sqn{g_i^{t}-\nabla f_i(x^\star)},
\end{align*}
so that
 \begin{align}
\Expc{ \sum_{i\in\Omega^t} \sqn{g_i^{t+1}-\nabla f_i(x^\star)}}&=\Expc{ \sum_{i=1}^n \sqn{g_i^{t+1}-\nabla f_i(x^\star)}}-  \frac{n-s}{n}\sum_{i=1}^n \sqn{g_i^{t}-\nabla f_i(x^\star)}.\label{eqa22}
\end{align}
Using \eqref{eql1}, \eqref{eqa21} and \eqref{eqa22} in \eqref{eq41}  yields
\begin{align*}
&\left(1+\frac{2\gamma\mu L}{L+\mu}\right)s\Expc{\sqn{x^{t+1}-x^\star}}+\left(1+\frac{2}{\gamma(L+\mu)}\right)\gamma^2
\Expc{ \sum_{i=1}^n \sqn{g_i^{t+1}-\nabla f_i(x^\star)}}\notag\\
&\leq s\sqn{ x^t -x^\star}+\left(\left(1+\frac{2}{\gamma(L+\mu)}\right)\frac{n-s}{n}+\frac{s}{n}\right)\gamma^2 \sum_{i=1}^n \sqn{   g_i^t -\nabla f_i(x^\star)}.
\end{align*}
Therefore,
 \begin{align}
\Expc{\Psi^{t+1}}&\leq \max\left\{\left(1+\frac{2\gamma\mu L}{L+\mu}\right)^{-1},\left(1+\frac{2}{\gamma(L+\mu)}\right)^{-1}\frac{s}{n}+\frac{n-s}{n}
\right\}\Psi^t\notag\\
&= \max\left\{\frac{L+\mu}{L+\mu + 2\gamma\mu L},\left(\frac{\gamma (L+\mu)}{\gamma(L+\mu)+2}-1\right)\frac{s}{n}+1
\right\}\Psi^t\notag\\
&= \max\left\{1-\frac{2\gamma\mu L}{L+\mu + 2\gamma\mu L},1-\frac{2}{\gamma(L+\mu)+2}\frac{s}{n}
\right\}\Psi^t.\label{eqff}
\end{align}
Using the tower rule, we can unroll this recursion to obtain the unconditional expectation of $\Psi^{t+1}$. Using classical results on supermartingale convergence \citep[Proposition A.4.5]{ber15}, it follows from \eqref{eqff} that $\Psi^t\rightarrow 0$ almost surely. Almost sure convergence of $x^t$ and the $g_i^t$ follows.

\section*{Acknowledgement}
This work was supported by the SDAIA-KAUST Center of Excellence in Data Science and Artificial Intelligence (SDAIA-KAUST AI).

\bibliographystyle{apalike}
\bibliography{IEEEabrv,biblio}

\end{document}